\documentclass[11pt]{article}
\usepackage{amsmath}
\usepackage{amssymb,amsbsy,amsthm}
\usepackage{graphicx}
\usepackage[dvipsnames]{xcolor}
\usepackage{enumerate}
\usepackage[margin=1in]{geometry}
\usepackage{hyperref}

\allowdisplaybreaks[1]

\numberwithin{equation}{section}

\renewcommand{\P}{\mathbb{P}} 
\newcommand{\E}{\mathbb{E}} 
\newcommand{\eps}{\varepsilon} 

\DeclareMathOperator{\poly}{poly} 
\DeclareMathOperator{\polylog}{polylog} %

\newcommand{\wh}{\widehat}  
\newcommand{\wt}{\widetilde} 

\newtheorem{theorem}{Theorem}[section]
\newtheorem{lemma}[theorem]{Lemma}

\begin{document}

\title{Trace reconstruction with $\exp( O( n^{1/3} ) )$ samples}
\author{
	Fedor Nazarov
	\thanks{Kent State University; \texttt{nazarov@math.kent.edu}.}
	\and
	Yuval Peres
	\thanks{Microsoft Research; \texttt{peres@microsoft.com}.}
	 }

\date{\today}

\maketitle


\begin{abstract}
In the trace reconstruction problem, an unknown bit string $x \in \{0,1\}^n$ is observed through  the deletion channel, which deletes each bit of $x$ with some constant probability $q$, yielding a contracted string $\widetilde{x}$. How many independent copies of $\widetilde{x}$ are needed to reconstruct $x$ with high probability? Prior to this work, the best upper bound, due to Holenstein, Mitzenmacher, Panigrahy, and Wieder (2008), was $\exp(\wt{O}(n^{1/2}))$.
We improve this bound to $\exp(O(n^{1/3}))$ using statistics of individual bits in the output
and show that this bound is sharp in the restricted model where this is the only information used.
Our method, that uses elementary complex analysis, can also handle insertions.
\end{abstract}


\section{Introduction} \label{sec:intro}


In the trace reconstruction problem, the goal is to reconstruct an unknown bit string $x \in \{0,1\}^n$ from multiple independent noisy observations of $x$.
We focus on the case where the noise is due to
 $x$ going through the deletion channel, where each bit is deleted independently with probability $q$ (and the remaining bits are concatenated, with no space between them, so the observer is uncertain about the original location of a bit in the output).
That is, instead of seeing $x$, we see (many independent copies of) $\wt{X}$, which is obtained as follows:
start with an empty string, and for $k = 0, 1, \dots, n-1$, do the following:
\begin{itemize}
 \item \textbf{(retention)} with probability $p=1 - q$, copy $x_k$ to the end of $\wt{X}$ and increase $k$ by $1$;
 \item \textbf{(deletion)} with probability $q$, increase $k$ by $1$.
\end{itemize}
Variants, such as including insertions and substitutions, are discussed in Section 5.

Given $T$ i.i.d.\ samples (known as {\em traces\/})  $\wt{X}^{1}, \dots, \wt{X}^{T}$, all  obtained from passing the same unknown string $x$  through the deletion channel, a trace reconstruction algorithm outputs an estimate
$\wh{X}$ which is a function of $  \wt{X}^{1}, \dots, \wt{X}^{T}$.   The main question is: Given $\delta>0$, how many samples are needed so that there is a choice of   $\wh{X}$ that satisfies  $\P_x[\wh{X} =x] \ge 1- \delta$ for every $x \in \{0,1\}^n$ ? \newline
(Here $\P_x$ is the law of $\wt{X}_{1}, \dots, \wt{X}_{T}$ when the original string was $x$.)
Prior to this work, the best available upper bound, due to \cite{HMPW08}, was $T=\exp(\wt{O}(n^{1/2}))$. Our main result, proved in Section 2, yields the following improvement.

\begin{theorem}\label{thm:main}
For any deletion probability $q < 1$ and any $\delta>0$, there exists a finite constant $C$ such that,  for any original string $x \in \{0,1\}^n$,
it can be reconstructed with   probability at least $1-\delta$ from
 $T = \exp \left( C n^{1/3} \right)$ i.i.d.\ samples of the deletion channel applied to $x$.
\end{theorem}

Our estimator will only use individual bit statistics from the outputs of the deletion channel.
The following Theorem, proved in Section 4,  shows that
$T = \exp \left( \Omega(n^{1/3}) \right)$ traces are needed for reconstruction if these are the only data used.

\begin{theorem}\label{thm:opt}
Fix a deletion probability $q < 1$.    For each $n$ there exist two distinct strings $x,y \in \{0,1\}^{n}$,
with the following property: For all $j$,   the total variation distance between the laws of $\Bigl(\wt{X}_{j}^{t}\Bigr)_{t=1}^T$
 and $\Bigl(\wt{Y}_{j}^{t}\Bigr)_{t=1}^T$ is at most $T\exp \left( -c n^{1/3}   \right)$, for some $c=c(q)>0$.
 \end{theorem}
 (Thus, for $c_1<c$, given
 $T = \exp \left( c_1 n^{1/3}   \right)$ i.i.d.\ samples of the $j$th bit from the output of the deletion channel, one cannot distinguish if they were generated from $x$ or from $y$.)

\subsection{Related work} \label{sec:related}

\begin{itemize}
\item
Batu, Kannan, Khanna, and McGregor~\cite{BKKM04} introduced the following algorithm, which they call Bitwise Majority Alignment (BMA),
to reconstruct a string from samples coming from the deletion channel.
The algorithm goes from left to right and reconstructs the string bit by bit.
For each sample, the algorithm maintains a pointer that initially points to the first bit of the sample.
To determine the next bit in the string, the algorithm simply looks at the bits that the pointers point to in each sample, and takes a majority vote, breaking ties arbitrarily.
The algorithm then moves the pointers to the right by one for those samples that agree with the majority vote
(hypothesizing that the sample had the correct bit).
For the samples which did not agree with the majority vote, the algorithm does not change the pointers
(hypothesizing that the mismatch is due to a deletion of this bit in this sample).

Batu~et~al.~\cite{BKKM04} prove that if the original string $x$ is random and the deletion probability
$q = O \left( 1 / \log n \right)$,
then $x$ can be reconstructed exactly with high probability using
$T = O \left( \log   n   \right)$ samples.
They also show that if
$q = O \left( n^{- \left( 1/2 + \eps \right)} \right)$,
then every $x$ can be reconstructed with high probability
with $T= O \left( n \log   n  \right)$ samples
using a variant of BMA.


\item
The results of Holenstein, Mitzenmacher, Panigrahy, and Wieder~\cite{HMPW08} are the state-of-the-art for the deletion channel.

For arbitrary input strings $x$, they show that $\exp \left( \sqrt{n} \polylog(n) \right)$ traces suffice for reconstruction for any constant deletion probability $q < 1$. For random $x$, they show that there exists a constant $\gamma$, such that if the deletion probability satisfies $q < \gamma$, then $\poly (n)$ traces suffice   to reconstruct $x$ in $\poly(n)$ time.

Their algorithm goes from left to right and determines each bit sequentially using voting. The main difference compared to BMA is that they do not allow all traces to vote. At each step, for a given trace they look back and see if it has the last $O(\log n )$ bits correct. If so, then they assume the pointer is in its right place and allow this trace to vote.

\item Elchanan Mossel (private communication) told us that in 2008, Mark Braverman, Avinatan Hassidim and Elchanan proved that, for some $c>0$, trace reconstruction algorithms relying on single bit statistics require $\exp(n^c)$ traces. That proof was not published.

\item
After the results of this paper were obtained, we learned that similar results were obtained independently and simultaneously by Anindya De, Ryan O'Donnell and Rocco Servedio.

\end{itemize}

\section{Proof of Theorem~\ref{thm:main}}


In the proof we consider the random power series
\begin{equation}\label{eq:stat}
\sum_{j \geq 0} \wt{a}_{j} w^{j},
\end{equation}
where $ \wt{{\mathbf a}}$ is a sample output of the deletion channel and $w \in \mathbb{C}$ is chosen appropriately.
The first lemma expresses the expectation of such a random series using the original sequence of interest.

\begin{lemma}\label{lem:key1}
Let $w \in \mathbb{C}$, let\/ ${\mathbf a} := \left( a_{0}, a_{1}, \dots, a_{n-1} \right) \in \mathbb{R}^{n}$, let\/ $ \wt{{\mathbf a}}$ be the output of the deletion channel with input ${\mathbf a}$, and pad\/ $ \wt{{\mathbf a}}$ with zeroes to the right. Write $p=1-q$. Then
\begin{equation}\label{eq:expectation}
\E \left[ \sum_{j \geq 0} \wt{a}_{j} w^{j} \right] = p \sum_{k=0}^{n-1} a_{k} \left( pw+q \right)^{k}.
\end{equation}
\end{lemma}
The proof is given in the next section. Intuitively speaking, this identity is useful because by averaging  samples we can approximate the expectation on the left-hand side of~\eqref{eq:expectation}, while from the right-hand side of~\eqref{eq:expectation} we can extract the original sequence ${\mathbf a} = \left( a_{0}, a_{1}, \dots, a_{n-1} \right)$.

Note that unless $\left| w \right| = 1$, either the first or last terms of $ \wt{{\mathbf a}}$ will dominate in the left-hand side of~\eqref{eq:expectation}.
Similarly, if we let $z := pw+q$, then unless $\left| z \right| = 1$, either the first or the last terms of ${\mathbf a}$ will dominate in the right-hand side of~\eqref{eq:expectation}.
We wish to give approximately equal weight to all terms, hence we would like $\left| w \right|$ and $\left| z \right|$ to both be close to 1.
This only happens if  both $w$ and $z$ are close to $1$; thus we will let $z$ vary along a small arc on the unit circle near $1$.
This explains our interest in the following lemma, which is a special case of Theorem 3.2 in~\cite{BE97}.

\begin{lemma}[Borwein and Erd{\'e}lyi~\cite{BE97}]  \label{lem:key2}
There exists a finite constant $c$ such that the following holds.
Let
$${\mathbf a} = \left( a_{0}, a_{1}, \dots, a_{n-1} \right) \in \left\{ -1, 0, 1 \right\}^{n}$$
be such that ${\mathbf a} \neq 0$.
Let $A \left( z \right) := \sum_{k = 0}^{n-1} a_{k} z^{k}$ and denote by $\gamma_L$   the arc $\left\{  e^{i \theta} : -\pi/L \leq \theta \leq  \pi / L \right\}$. Then $\max_{z \in \gamma_L} |A(z)| \geq e^{-cL}$.
\end{lemma}

We will optimize over the length of the arc $\gamma_L$, and in the end we shall choose $L$  of order $  n^{1/3}$.

Note that if  $z$ is in the arc $\gamma_L=\left\{ e^{i \theta} : -\pi / L  \leq \theta \leq  \pi / L \right\}$
and , then
\begin{equation} \label{wbound}
 w = (z-q)/p \quad \mbox{\rm satisfies} \; | w | \le \exp \left( C_1 / L^{2} \right) \,
\end{equation}
for some constant $C_1=C_1(q)$. This is because writing
$z = \cos \theta + i \sin \theta$,
and using the Taylor expansion of cosine, we get
\begin{eqnarray*}
\left| w \right|^{2} &=& \frac{1+q^2-2q\cos(\theta)}{p^2} = \frac{1+q^2-2q+2q(1-\cos \theta)}{p^2} \\  &\le& 1 + \frac{q}{p^2} \theta^{2} + O \left( \theta^{4} \right)
 = \exp \left( q \theta^{2}/p^2 + O \left( \theta^{4} \right) \right)\, .
\end{eqnarray*}
The quadratic term $\theta^{2}$ is to be expected: when $z$ is on the unit circle, $w=(z-q)/p$ is on a circle of radius $1/p$ centered at $-q/p$; these circles are tangent at 1.



\subsection{Proof of Theorem~\ref{thm:main} using the lemmas}

Let $x,y \in \left\{ 0, 1 \right\}^{n}$ be two different bit sequences. Our first goal is to distinguish between $x$ and $y$.
Let ${\mathbf a} := x-y$ and let $A(z) := \sum_{k = 0}^{n-1} a_{k} z^{k}$.
Given a large integer  $L$ (which we shall choose later), fix $z$ in the arc 
$\gamma_L=\left\{e^{i \theta} : -\pi/L \leq \theta \leq \pi / L \right\}$ such that
$\left| A (z) \right| \geq e^{-cL} $;   
such a $z$ exists by Lemma~\ref{lem:key2}.
Let $w = (z-q)/p$. Recall from the previous subsection that $\left| w \right| < \exp \left( C_1 / L^{2} \right)$ for some $C_1 < \infty$.

Considering the random series defined in~\eqref{eq:stat}, we see via Lemma~\ref{lem:key1} that
\[
\E \Bigl[ \sum_{j \geq 0} \wt{X}_{j} w^{j} \Bigr] - \E \Bigl[ \sum_{j \geq 0} \wt{Y}_{j} w^{j} \Bigr] = A (z)\, .
\]
Taking absolute values,
\[
 \sum_{j \geq 0} \Bigl|\E \Bigl[ \wt{X}_{j}-\wt{Y}_{j}   \Bigr] \Bigr| \cdot |w|^{j}  \ge |A (z)| \ge e^{-cL},   
\]
whence by (\ref{wbound}),
\[
 \sum_{j \geq 0} \Bigl|\E \Bigl[ \wt{X}_{j}-\wt{Y}_{j}  \Bigr] \Bigr|  \ge \exp \Bigl( -C_1 n / L^{2} \Bigr)  \cdot e^{-cL} \,.
\]
To approximately maximize the right-hand side, we choose $L$ to be the integer part of $n^{1/3} $ and obtain that for some constant $C_2$,
\[
 \sum_{j \geq 0} \Bigl|\E \Bigl[ \wt{X}_{j}-\wt{Y}_{j}  \Bigr] \Bigr|      \ge \exp \Bigl(- C_2 n^{1/3} \ \Bigr) \,.
\]
We infer that there must exist some smallest $j<n$ for which
\begin{equation} \label{meanbound}
\Bigl|\E \Bigl[ \wt{X}_{j}-\wt{Y}_{j}  \Bigr] \Bigr| \ge \frac{1}{n} \exp \Bigl(- C_2 n^{1/3}   \Bigr) \,.
\end{equation}
We denote this choice of $j$ by $j(x,y)$.
Now suppose that $u$ is either $x$ or $y$ and we observe $T$ i.i.d.\ samples $\wt{U}^{1}, \ldots \wt{U}^T$ of the deletion channel applied to $u$.
We say that $y$ {\bf beats} $x$ (with respect to these samples) if  for $j=j(x,y)$ we have
$$\Bigl| \frac{1}{T} \sum_{t=1}^T \wt{U}^{t}_j- \E_y [ \wt{Y}_{j}] \Bigr|  \le \Bigl| \frac{1}{T} \sum_{t=1}^T \wt{U}^{t}_j- \E_x [ \wt{X}_{j}] \Bigr| \,.$$
The random bits $\wt{U}^{t}_j$ for $t=1,\ldots$ are i.i.d.; their mean is $ E_x [ \wt{X}_{j}]$ if $u=x$, and is $\E_y [ \wt{Y}_{j}]$ if $u=y$.
The difference of these means is at least $\eta=\frac{1}{n}\exp \Bigl(- C_2 n^{1/3}   \Bigr)$.
Thus by the standard Chernoff bound (or Hoeffding's inequality~\cite{hoeffding}),
$$
\P_x\Bigl[ y  \, \mbox{ beats } \,  x \Bigr] \le \exp(-T\eta^2/2)=\exp\Bigl(-\frac{T}{2n^2} \exp \bigl(- 2C_2 n^{1/3}    \bigr) \Bigr) \,.
$$
Given the deletion channel outputs, we define $\wh{X}=x$ if no string $y \ne x$ beats $x$. Observe that there can be at most one such unbeaten $x$.
If there is no such unbeaten string, define $\wh{X}$ arbitrarily.
Then
\begin{equation} \label{union}
\P_x[\wh{X} \ne x] \le \sum_{y \ne x} \P_x\Bigl[ y  \, \mbox{ beats } \,  x \Bigr]
\le 2^n  \exp\Bigl(-\frac{T}{2n^2} \exp \bigl(- 2C_2 n^{1/3}  \bigr) \Bigr) \,.
\end{equation}
Taking $T=\exp\Bigl(C_3 n^{1/3}   \Bigr)$ for   $C_3 >2C_2$ makes the right-hand side of (\ref{union}) tend to 0.
\hfill $\Box$

\section{Proof of the polynomial identity and a simplified inequality}

\begin{proof}[Proof of Lemma~\ref{lem:key1}]
For $j \leq n-1$, the output bit $\wt{a}_{j}$ must come from an input bit $a_{k}$ for some $k \ge j$.
Now $\wt{a}_{j}$ comes from $a_{k}$ if and only if exactly $j$ among $a_{0}, a_{1}, \dots a_{k-1}$ are retained and $a_{k}$ is also retained.
There are $\binom{k}{j}$ ways of choosing which $j$ bits among $a_{0}, a_{1}, \dots a_{k-1}$ to retain, and the probability of each such choice is $p^j q^{k-j}$.
The probability of retaining $a_{k}$ is $p$. Putting everything together, we  obtain that
\[
\E \Bigl[ \sum_{j \geq 0} \wt{a}_{j} w^{j} \Bigr] = p \sum_{j \geq 0} w^{j} \sum_{k = j}^{n-1} a_{k} \binom{k}{j} p^j q^{k-j}  \,.
\]
Changing the order of summation, we infer that
\[
\E \Bigl[ \sum_{j \geq 0} \wt{a}_{j} w^{j} \Bigr] = p \sum_{k = 0}^{n-1} a_{k} \sum_{j=0}^{k} \binom{k}{j} p^j q^{k-j}  w^{j}.
\]
Finally, observe that the sum over $j$ on the right-hand side is exactly the binomial expansion of $(pw+q)^{k}$.
\end{proof}

Since the proof of  Lemma~\ref{lem:key2} in \cite{BE97} is somewhat involved, for expository purposes, we prove here a weaker estimate. This   is simpler to prove and it does not result in a much weaker conclusion. Specifically, if we use Lemma~\ref{lem:key2_weak} below as a black box instead of Lemma~\ref{lem:key2}, then we obtain  that $T = \exp \Bigl( c n^{1/3} \log n \Bigr)$ samples suffice for trace reconstruction; comparing this with Theorem~\ref{thm:main}, we only lose a log factor in the exponent.

\begin{lemma}\label{lem:key2_weak}
Let
${\mathbf a} = \Bigl( a_{0}, a_{1}, \dots, a_{n-1} \Bigr) \in \Bigl\{ -1, 0, 1 \Bigr\}^{n}$
be such that ${\mathbf a} \neq 0$.
Let $A (z) := \sum_{k = 0}^{n-1} a_{k} z^{k}$.
If $ | A (z)| \leq \lambda$ on the arc $\gamma_L:=\Bigl\{ z = e^{i \theta} : -\pi / L \leq \theta \leq  \pi / L \Bigr\}$,
then $\lambda \geq n^{- L}$.
\end{lemma}

\begin{proof}
We may assume w.l.o.g.\ that $a_{0} = 1$. (Indeed, if $a_{m}$ is the first nonzero entry and $m \geq 1$, then replace $A(z)$ by $A(z)/z^{m}$; this does not change the magnitude of the function on the unit circle, and yields a polynomial with $a_{0} \neq 0$. Multiplying $A(z)$ by an appropriate sign, we can guarantee that $a_{0} = 1$.) In other words, $A(0) = 1$.
 
Consider the product
\begin{equation}\label{eq:A_rotations}
{F} (z) := \prod_{j=0}^{L-1} A \Bigl( z \cdot e^{2 \pi i j / L} \Bigr).
\end{equation}
We again have that $F(0) = 1$. By the maximum principle, the the maximum absolute value of the polynomial $F(z)$ on the unit disc is attained on the boundary, i.e., on the unit circle.
Thus there exists $z$ such that $ | z  | = 1$ and $\Bigl| F  (z) \Bigr| \geq 1$.
On the other hand, for every $z$ such that $ | z  | = 1$, the assumption of the lemma guarantees that there is at least one factor in~\eqref{eq:A_rotations} whose absolute value is at most $\lambda$.
Using the trivial bound $ |A (z)| \leq n$ for every other factor,
we obtain that $ | F (z)  | \leq \lambda n^{L-1}$ for every $z$ such that $ | z  | = 1$.
Putting the two inequalities together we obtain that
$\lambda \geq n^{-(L-1)}$.
\end{proof}

\section{Optimality for single bit tests} 
\begin{proof}[Proof of Theorem~\ref{thm:opt}] 
Let $L:=n^{1/3}$. (To keep the notation light, we omit integer parts and use $c_j$ to denote  absolute constants and constants that depend only on $q$). By Theorem 3.3 in~\cite{BEK99}, there exists a polynomial $Q$ of degree $c_2 L^2$,  with coefficients in $\{-1,0,1\}$,
such that 
$$\max_{z \in [0,1]} |Q(z)| \le \exp(-c_3 L) \,.
$$ 
 Write $Q$  in the form $Q=\varphi -\psi $
where $\varphi$ and $\psi$ are polynomials of degree $c_2 L^2$ with coefficients in $\{0,1\}$.

Let $\widehat{E}_L$ denote the ellipse with foci at $1-8/L$ and $1$ and with major axis $[1-14/L,1+6/L]$, i.e.,
$$
\widehat{E}_L=\{z: |z-(1-8/L)|+|z-1| \le 20/L \} \,.
$$
 As explained on page 11 of~\cite{BE97},
Corollary 4.5 of that paper implies that
$$
\max_{z \in \widehat{E}_L} |Q(z)| \le e^{-c_4 L} \, .
$$
Recall that $p=1-q$ and let $\Gamma$ denote the circle $\{z: |z-q|=p\}$. Then $\Gamma$ intersects the ellipse $\widehat{E}_L$ in an arc $\Gamma_L$ of length $c_5/L$, since $\widehat{E}_L$ contains the disk of radius $6/L$ centered at 1.
Thus we may write
$$\Gamma_L=\{pe^{i\theta}+q : -c_6/L \le \theta \le c_6/L \} \,.
$$

Let $m:=(n-c_2L^2)/2$. Define the string $x \in \{0,1\}^{n}$ where the first $m$ bits are zeros, the next $c_2L^2$ bits are the coefficients of $\varphi$, and the final $m$ bits are zeros. The string $y \in \{0,1\}^{n}$ is constructed from $\psi$ in the same way.
 Then 
$$
A(z):=\sum_{k=0}^{n-1} (x_j-y_j)z^j=z^mQ(z) 
$$
satisfies 
\begin{equation} \label{maxA}
\max_{z \in \Gamma_L} |A(z)| \le e^{-c_4 L} \,.
\end{equation}
Define $b_j:=\E \Bigl[ \wt{X}_{j}-\wt{Y}_{j}]$ and $B(w):=\sum_{j=0}^{n-1} b_j w^j$. By Lemma~\ref{lem:key1}, we have
$B(w)=pA(pw+q)$.  We can extract $b_j$ from $B(\cdot)$ by integration:
$$
b_j=\frac{1}{2\pi}\int_{-\pi}^{\pi} e^{-ij\theta}  B(e^{i\theta})\,d\theta \,.
$$
Therefore
\begin{equation} \label{eq:b_j}
|b_j| \le  \frac{1}{2\pi}\int_{-\pi}^{\pi}    |B(e^{i\theta})|\, d\theta \le \frac{1}{2\pi}\int_{-\pi}^{\pi}   |A(pe^{i\theta}+q)|\, d\theta \,.
\end{equation}
For $\theta \in [-c_6/L,c_6/L]$, the integrand on the right-hand side is at most $e^{-c_4 L}$ by (\ref{maxA}).
To bound that integrand for larger $\theta$, observe that 
\begin{eqnarray} \label{inside}
|pe^{i\theta}+q|^2 &=&  p^2\cos^2\theta+2pq\cos\theta+q^2+p^2\sin^2\theta=(p+q)^2+2pq(\cos\theta-1) \\
&=&  1-pq\theta^2+O(\theta^4)  \le 1-c_7 \theta^2 \,.
\end{eqnarray}
Since $|A(z)| \le |z|^m (1-|z|)^{-1}$ in the unit disk and $m>n/3$, we infer that if $|\theta|>c_6/L$, then
$$
|A(pe^{i\theta}+q)| \le c_8 L^2 (1-c_9 L^{-2})^{n/3} \le \exp(-c_{10} nL^{-2}) =e^{-c_{10} L} \,.
$$
In conjunction with (\ref{maxA}), we conclude that the integrand on the right-hand side of (\ref{eq:b_j}) is uniformly bounded
by $e^{-c_{11} L}$, whence
\begin{equation} \label{eq:b_j2}
|b_j| \le  e^{-c_{11} L} \quad \mbox{\rm for all } \, j \,.
\end{equation}

Next, fix $j$. To bound the total variation distance between the laws of $\bigl(\wt{X}_{j}^{t}\bigr)_{t=1}^T$
 and $\bigl(\wt{Y}_{j}^{t}\bigr)_{t=1}^T$, we will use a greedy coupling. More precise estimates can be obtained, e.g., using Hellinger distance, but the improvement  will not affect the final result.  Let $\bigl(\xi_t \bigr)_{t=1}^T$ be i.i.d. variables, uniform in $[0,1]$.
 Then ${\bf 1}_{\xi_t \le  \E(\wt{X}_{j})}$ has the law of $\wt{X}_{j}^t$ and ${\bf 1}_{\xi_t \le  \E(\wt{Y}_{j})}$ has the law of $\wt{Y}_{j}^t$.
 These indicators differ with probability $b_j$. Altogether, this coupling implies that the total variation distance between the laws of $\bigl(\wt{X}_{j}^{t}\bigr)_{t=1}^T$  and $\bigl(\wt{Y}_{j}^{t}\bigr)_{t=1}^T$ is at most $Tb_j$.
 Referring to (\ref{eq:b_j2}) concludes the proof.
\end{proof}

\noindent{\bf Remark.} Strictly speaking, padding $x$ and $y$ with zeros on the right  was not really needed in the above proof.
The reason for it is that one can also consider single bit tests on the traces using the $j$th bit from the right in each output;
the additional padding and a symmetry argument ensures that these tests will also require $\exp(\Omega(n^{1/3})$ traces for reconstruction.

\section{Substitutions and insertions}
If, after $x$ goes through the deletion channel with deletion probability $q$, every bit is flipped with probability $\lambda<1/2$,
then $\exp(O(n^{1/3})$ samples still suffice for reconstruction. Indeed, let   $X^\# $ be the output of this deletion-substitution channel with input $ x$, padded with zeroes to the right. Define $Y^\#$ from $y$ similarly.   Recall that $p=1-q$. Then $\E(X^\#_{j}-Y^\#_j)=(1-2\lambda)\E(\wt{X}_{j}-\wt{Y}_j)$, so
(\ref{eq:expectation}) is replaced by
\begin{equation}\label{eq:expectation2}
\E \left[ \sum_{j \geq 0}  (X^\#_{j}-Y^\#_j) w^{j} \right] = (1-2\lambda) p \sum_{k=0}^{n-1} (x_k-y_k) \left( pw+q \right)^{k}.
\end{equation}
The analysis in Section 2 then proceeds without change, since the pre-factor $1-2\lambda$ is immaterial.

\smallskip

{\bf Insertions} are more interesting. Suppose that before each bit $x_k$ in the input, $G_k-1$ i.i.d.\ fair bits are inserted, where the variables
$G_k$ are i.i.d.\ with a Geometric$(\alpha)$ distribution, i.e., denoting $\beta=1-\alpha$, for all $\ell \ge 1$,
$$
\P(G_k=\ell)=\alpha\beta^{\ell-1} \,.
$$
After $x_{n-1}$, at the end of the sequence, $G_{n}-1$ additional fair bits are appended. We call $\beta=1-\alpha$ the insertion parameter.
Thus, such an insertion channel will yield an output $X^*$ consisting of $G_0-1$ i.i.d.\ fair bits, followed by $x_0$, followed by $G_1-1$ i.i.d.\ fair bits, followed by $x_1$, etc., ending with $x_{n-1}$ and the $G_{n}-1$ bits after it. The next theorem is analogous to Theorem~\ref{thm:main}.

\begin{theorem}\label{thm:insert}
For any insertion parameter $\beta < 1$ and any $\delta>0$, there exists a finite constant $C$ such that,  for any original string $x \in \{0,1\}^n$,
it can be reconstructed with   probability at least $1-\delta$ from
 $T = \exp \left( C n^{1/3} \right)$ i.i.d.\ samples of the insertion channel applied to $x$.
\end{theorem}
\begin{proof}
To prove this theorem, an analog of Lemma~\ref{lem:key1} is needed:

\begin{lemma}\label{lem:keyins}
Given strings $x$ and $y$ in $\{0,1\}^n$, let $X^*$ and $Y^*$ denote the corresponding outputs of the insertion channel with parameter $\beta$, where $\alpha+\beta=1$.  Then for $w \in \mathbb{C}$, we have
\begin{equation}\label{eq:expectation3}
\E \left[ \sum_{j \geq 0}  (X^*_j-Y^*_j) w^{j+1} \right] =  \sum_{k=0}^{n-1} (x_k-y_k) \Bigl(\frac{\alpha w}{1-\beta w}  \Bigr)^{k+1} \,.
\end{equation}
\end{lemma}
We also need an analog of (\ref{wbound}): if 
$$ 
\zeta=e^{i\theta}=\frac{\alpha w}{1-\beta w}
$$ 
is on the unit circle, then 
\begin{equation}\label{analog}
w=\zeta/(\alpha+\beta \zeta) \quad \mbox{\rm satisfies} \;  |w|^2 \le 1+C\theta^2 \,,
\end{equation}
 for some constant $C=C(\alpha)$. This is immediate from (\ref{inside}), with $\alpha,\beta$ replacing $q,p$ there.

With Lemma~\ref{lem:keyins} and the inequality (\ref{analog}) in hand, the rest of the proof of Theorem~\ref{thm:insert} is identical to the proof of Theorem~\ref{thm:main}. 
\end{proof}

\begin{proof}[Proof of Lemma~\ref{lem:keyins}]
It is convenient to couple $X^*$ and $Y^*$ to use the same geometric variables $G_0,G_1,\ldots$ and the same inserted bits. Note that the choice of coupling does not affect $\E(X^*_{j}-Y^*_j)$. Write $a_k=x_k-y_k$ and $D_j:=X^*_{j}-Y^*_j$. Then
$$
\E(D_j)=\sum_{k=0}^j \P(G_0+\dots + G_k=j+1)a_k=\sum_{k=0}^j {j \choose k} \alpha^{k+1} \beta^{j-k} a_k \,.
$$
Therefore, using the classical expansion
 $$
\sum_{j=k}^\infty {j \choose k} s^{j-k} =(1-s)^{-k-1}
$$
we obtain that 
\begin{eqnarray*}
\E\left[ \sum_{j \geq 0}  D_j w^{j+1} \right] &=&   \sum_{j \geq 0}  \sum_{k=0}^j w^{j+1}  {j \choose k} \alpha^{k+1} \beta^{j-k} a_k \\
&=& \sum_{k \ge 0} a_k (\alpha w)^{k+1} \sum_{j \geq k} {j \choose k} (\beta w)^{j-k} \\
&=& \sum_{k \ge 0} a_k (\alpha w)^{k+1} (1-\beta w)^{-k-1} \,,
\end{eqnarray*}
which is the same as (\ref{eq:expectation3}).
\end{proof} 

\noindent{\bf Remark.} To combine deletions, insertions and substitutions, simply compose the linear transformation $w \mapsto pw+q$
that appears in (\ref{eq:expectation2}) with the M\"obius transformation $w \mapsto \alpha w/(1-\beta w)$. Each of these transformations maps the unit circle
to a smaller circle that is tangent to it at 1, and this also holds for their composition, in any order.


\section*{Acknowledgements}
We first learned of the trace reconstruction problem from Elchanan Mossel and Ben Morris. The second author is grateful to them, as well as to Ronen Eldan, Robin Pemantle and Perla Sousi for many discussions of the problem. The insightful suggestion by Elchanan that one should focus on deciding between
two specific candidates for the original bit string was particularly influential. We are also indebted to Miki Racz and Gireeja Ranade  for their help
with the exposition.


\bibliographystyle{plain}
\bibliography{tr}




\end{document}